\documentclass[12pt]{amsart}
\usepackage{amsmath}
\usepackage{amsfonts}
\usepackage{amssymb}
\usepackage{amsthm}
\usepackage{enumitem}
\usepackage{todonotes}
\usepackage{hyperref}

\newtheorem{theorem}{Theorem}
\newtheorem{proposition}[theorem]{Proposition}
\newtheorem{lemma}[theorem]{Lemma}
\newtheorem{kor}[theorem]{Corollary}
\theoremstyle{definition}
\newtheorem*{example}{Example}
\newtheorem*{rem}{Remark}

\begin{document}
\title[On Diophantine pairs and triples of triangular numbers]{On Diophantine pairs and triples of triangular numbers}
\author{Marija {Bliznac Trebje\v{s}anin}}
\address{Faculty of Science, University of Split, Ru\dj{}era Bo\v{s}kovi\'{c}a 33, 21000 Split, Croatia}
\email{marbli@pmfst.hr}

\begin{abstract}
We investigate Diophantine pairs and triples of triangular numbers with the property $D(a)$ for a non-zero integer $a$. We prove that if a triangular number belongs to a $D(a)$-pair of triangular numbers, then it belongs to infinitely many $D(a)$-triples of triangular numbers. Additionally, we determine infinite families of integers $a$ that admit such pairs, as well as families for which no $D(a)$-pairs can exist.
\end{abstract}

\maketitle


\section{Introduction}

Let $a\neq0$ be an integer. We call a set of $m$ positive integers a $D(a)$-$m$-tuple (or a Diophantine $m$-tuple with the property $D(a)$) if the product of any two of its distinct elements increased by $a$ is a perfect square.
When $a=1$, the classical case introduced by Diophantus, such a set is simply referred to as a Diophantine $m$-tuple.

The study of such sets traces back to Fermat, who discovered the first Diophantine quadruple $\{1, 3, 8, 120\}$. A central question in this area of research is determining the maximum possible size of these sets. It is now established that $D(1)$-quintuples \cite{htz}, $D(4)$-quintuples \cite{petorke_btf}, $D(-4)$-quadruples \cite{bcm} and $D(-1)$-quadruples \cite{bcm} do not exist. Another prominent direction of research involves restricting the elements of $D(a)$-$m$-tuples to specific integer sequences, such as Fibonacci numbers (see, for instance, \cite{duje_gen, fujita_luca_fib}). For a comprehensive overview of various generalizations and related problems, we refer the reader to \cite[Section 1.5]{duje_book}.

For a positive integer $n$, the $n$-th triangular number, denoted by $T_n$, is defined as the sum of the first $n$ positive integers, given by the explicit formula
$$T_n=\frac{n(n+1)}{2}.$$

From the main result of \cite{mnde}, it follows that
$\{T_n,T_{n+4},T_{4n^2+20n+8}\}$
is a Diophantine triple of triangular numbers. In \cite{mbt_triangular}, the author presented a construction showing that for any positive integers $k$ and $n$, the triangular number $T_n$ is a member of infinitely many $D(k^2)$-triples containing only triangular numbers.

Let $a\neq 0$ and $n\geq 1$ be integers. We are looking for a positive integer $m$ and a nonnegative integer $r$ satisfying
$$T_nT_m+a=r^2.$$
This equation can be rewritten as a generalized Pellian equation
\begin{equation}\label{pellova_x}
    x^2-n(n+1)y^2=16a-n(n+1),
\end{equation}
where the unknowns are $x=4r$ and $y=2m+1$. An analysis of the properties of this equation will provide the main results concerning the existence of $D(a)$-pairs and $D(a)$-triples of triangular numbers.

The existence of a solution to equation \eqref{pellova_x}, with an odd integer $y$, depends on properties of both $a$ and $n$. As shown in \cite{mbt_triangular}, for $a=k^2$, $k\in\mathbb{N}$, we have a solution for every integer $n\geq 1$. Determining the conditions on $a$ and $n$ under which a solution exists will be the primary focus of Sections \ref{section:ex_pairs} and \ref{section:member_pairs}. In the general case for an integer $a\neq 0$, by analyzing the sequences $(m_k)_{k \in \mathbb{N}}$ arising from the solutions to equation \eqref{pellova_x}, we will prove that for any two successive terms of the sequence, the set $\{T_{m_k}, T_{m_{k+1}}\}$ forms a $D(a)$-pair. This naturally leads to our main result.

\begin{theorem}\label{tm_2}
Let $a\neq 0$ and $n\geq 1$ be integers such that there exists a positive integer $m$ for which $\{T_n,T_m\}$ is a $D(a)$-pair. Then $T_n$ is an element of infinitely many $D(a)$-triples of triangular numbers.
\end{theorem}

The remainder of this paper is organized as follows. In Section \ref{section:triple}, we cover the basic properties of the associated generalized Pellian equation and prove Theorem \ref{tm_2}. In Section \ref{section:ex_pairs}, we explore the conditions on the integer $a$ for which $D(a)$-pairs can or cannot exist, providing infinite families for both cases. Finally, in Section \ref{section:member_pairs}, we investigate the necessary conditions an integer $n$ must satisfy for $T_n$ to belong to a $D(a)$-pair, with a specific focus on the case $a=-1$, and describe an algorithm to generate all such values up to a given bound.

\section{Generalized Pellian equation and existence of $D(a)$-triples}\label{section:triple}

Let $a\neq 0$ and $n\geq 1$ be integers. We are searching for a positive integer $m$ and a nonnegative integer $r$ such that $T_n \cdot T_m +a= r^2$, or more precisely,
$$\frac{n(n+1)}{2} \cdot \frac{m(m+1)}{2}+a = r^2.$$
After some algebraic manipulation, this yields the generalized Pellian equation
\begin{equation}\label{pellova_m}
    (4r)^2 - n(n+1)(2m+1)^2 = 16a-n(n+1),
\end{equation}
which can be written in the simpler form of the equation \eqref{pellova_x} in terms of the variables $x$ and $y$.

From the theory of generalized Pellian equations, we know that if equation \eqref{pellova_x} has a solution, it has a finite number of classes of solutions. As described in \cite[Theorems 108 and 108a]{nagell}, each class is generated by a fundamental solution $(x^*,y^*)$ satisfying the inequalities
\begin{equation}\label{ineq_granice_fund}
0\leq y^*\leq \frac{Y_1}{\sqrt{2(X_1-1)}}\sqrt{|N|},\quad 0\leq|x^*|\leq\sqrt{\frac{1}{2}(X_1+1)\cdot |N|},
\end{equation}
where $N=16a-n(n+1)$ and $(X_1,Y_1)=(2n+1,2)$ is the fundamental solution of the corresponding Pell equation
\begin{equation}\label{eq_pellova_1}
X^2-n(n+1)Y^2=1.
\end{equation}
We can assume $N\neq 0$, since for $N=0$ equation \eqref{pellova_x} has only the trivial solution $(x,y)=(0,0)$.

We consider a specific class of solutions to equation \eqref{pellova_x} and denote its fundamental solution by $(x^*,y^*)$. If $N>0$, we assume that $x^*>0$. For the proof of Theorem~\ref{tm_2} this is no restriction: if $x^*<0$, we consider instead the conjugate class, whose fundamental solution is $(-x^*,y^*)$, since the map $(x,y)\mapsto(-x,y)$ maps a class onto its conjugate class and preserves $y$. Then every positive integer solution $(x,y)$ belonging to this class can be expressed as $(x,y)=(x_k,y_k)$ for some integer $k\geq 0$ if $x^*>0$, and $k\geq 1$ if $x^*\leq 0$, where the sequences $(x_k)_{k\geq 0}$ and $(y_k)_{k\geq 0}$ are defined by
$$x_{k}+y_{k}\sqrt{n(n+1)}=(x^*+y^*\sqrt{n(n+1)})\left(X_1+Y_1\sqrt{n(n+1)}\right)^k.$$

For the solutions $(x_k,y_k)$ of the generalized Pellian equation  \eqref{pellova_x}, the following relation holds:
\begin{equation}\label{niz_y_x}
    y_{k+1}=X_1y_k+Y_1x_k=(2n+1)y_k+2x_k,\quad k\geq 0.
\end{equation}

\begin{lemma}\label{lema_svojstva_xy}
    If $(x,y)$ is a solution to the equation \eqref{pellova_x} with $y$ an odd number, then  $x$ is divisible by $4$ and $y^*\geq 1$ is also an odd number, where $(x^*,y^*)$ is the fundamental solution of the class containing $(x,y)$. Moreover, if $a$ is not a perfect square, then $y^*\geq 3$.
\end{lemma}

\begin{proof}
Equation \eqref{pellova_x} can be rewritten as $x^2 - 16a = n(n+1)(y^2 - 1)$.

Since $y$ is odd, $y^2 - 1$ is a multiple of $8$ and $n(n+1)(y^2 - 1)$ is divisible by $16$, implying that $x$ is divisible by $4$.
By \eqref{niz_y_x}, $y_{k+1}\equiv y_k\pmod 2$, and the same holds for multiplication by $X_1-Y_1\sqrt{n(n+1)}$, which gives $(2n+1)y-2x$, and for the map $(x,y)\mapsto(-x,-y)$. Consequently, all solutions in the same class have $y$ of the same parity, so $y^*$ is odd as well.

Finally,  the assumption $y^{*}=1$ leads to $x^{*2} = 16a$, implying that $a$ is a perfect square.
\end{proof}

For $a=-1$ it holds that equations \eqref{pellova_m} and \eqref{pellova_x} are equivalent, meaning that the positive integer solutions $(x,y)$ of equation \eqref{pellova_x} are always of the form $(4r,2m+1)$ for some positive integers $r$ and $m$.

\begin{kor}\label{kor_svojstva_xy_minus_1}
    Let $a=-k^2$, where $k$ is a positive integer which has no prime divisor congruent to $3$ modulo $4$. If $(x,y)$ is a solution to the equation \eqref{pellova_x} then $y$ is an odd number and $x$ is divisible by $4$. Moreover, if $y$ is a positive integer, then $y\geq 3$.
\end{kor}
\begin{proof}
Equation \eqref{pellova_x} can be rewritten as $x^2 + 16k^2 = n(n+1)(y^2 - 1)$, implying $y\neq 0$ and $x$ is an even number. Denote $x=2z$ for an integer $z$. 
Suppose that $y$ is even and denote $y=2c$ for a nonzero integer $c$. Substituting these into the equation yields 
$4(z^2 + 4k^2) = n(n+1)(4c^2 - 1).$
Since $4c^2 - 1 \equiv 3 \pmod{4}$ and $4c^2-1>0$, it must have at least one prime factor $q \equiv 3 \pmod{4}$. As $q$ is odd, we get $q \mid (z^2 + 4k^2)$, which implies $z^2 \equiv -4k^2 \pmod{q}$. Since $q\nmid k$ by the assumption on $k$, this contradicts the fact that $-4$ is a quadratic non-residue modulo any prime $q \equiv 3 \pmod{4}$. Thus, $y$ must be odd and $x$ must be divisible by $4$ by Lemma \ref{lema_svojstva_xy}.

The assumption $y=1$ leads to $x^2 = -16k^2$, a contradiction.
\end{proof}

The sequence of solutions for the second coordinate, $(y_k)_{k \geq 0}$, forms a linear recurrence sequence
\begin{align}\label{eq_rek_y}
&y_0=y^*,\quad y_1=2x^*+(2n+1)y^*,\nonumber\\
&y_{k}=2(2n+1)y_{k-1}-y_{k-2},\ k \geq 2.
\end{align}
Since $y_k=2m_k+1$, we can express the sequence of indices $(m_k)$ such that $T_nT_{m_k}+a$ is a perfect square as follows:
\begin{align}\label{eq_niz_mk}
&m_0=\frac{1}{2}(y^*-1),\quad m_1=\frac{1}{2}(2x^*+(2n+1)y^*-1),\nonumber\\
 &m_{k}=2(2n+1)m_{k-1}-m_{k-2}+2n,\quad k\geq 2.
\end{align}

\begin{lemma}\label{lem_y}
    For the sequence $(y_k)_{k \geq 0}$ of solutions to the generalized Pellian equation \eqref{pellova_x}, the following identity holds:
$$y_{k+1}^2 - 2(2n+1)y_k y_{k+1} + y_k^2 =  64a - 4n(n+1).$$
\end{lemma}
\begin{proof}
From equation \eqref{niz_y_x}, we can express $2x_k$ as $2x_k = y_{k+1} - (2n+1)y_k$. Squaring this identity and applying the relation $x_k^2=n(n+1)(y_k^2-1)+16a$ yields the desired equality.
\end{proof}

Now we are ready to prove that $\{T_{m_k},T_{m_{k+1}}\}$ forms a $D(a)$-pair for all $k\geq 1$.

\begin{proposition}\label{prop_mk}
    Let $(m_k)_{k\geq 1}$ be the sequence defined in \eqref{eq_niz_mk}. Then
    $$T_{m_k}T_{m_{k+1}}+a$$
    is a square of a positive integer.
\end{proposition}
\begin{proof}
    Notice that $8T_n+1=(2n+1)^2$ holds for every integer $n$. Then
    \begin{align*}
        T_{m_k}T_{m_{k+1}}+a&=\frac{(2m_k+1)^2-1}{8}\cdot\frac{(2m_{k+1}+1)^2-1}{8} +a\\
        &=\frac{y_k^2-1}{8}\cdot\frac{y_{k+1}^2-1}{8} +a=\frac{y_k^2y_{k+1}^2-y_k^2-y_{k+1}^2+1+64a}{64}.
    \end{align*}
    By substituting the identity from Lemma \ref{lem_y}, we obtain
    \begin{align*}
        T_{m_k}T_{m_{k+1}}+a
        &=\frac{y_k^2y_{k+1}^2+4n(n+1)-64a-2(2n+1)y_ky_{k+1}+1+64a}{64}\\
        &=\left(\frac{y_ky_{k+1}-(2n+1)}{8}\right)^2.
    \end{align*}
    It remains to show that the expression inside the parentheses is a positive integer. It is easy to prove by mathematical induction that $y_ky_{k+1}-(2n+1)$ is divisible by $8$ for each $k\geq 0$. Properties from Lemma \ref{lema_svojstva_xy} are used in the proof.
    Moreover, $x_k>0$ and $y_k>0$ for $k\geq 1$ (see the remark at the end of this section), so \eqref{niz_y_x} gives $y_ky_{k+1}=(2n+1)y_k^2+2x_ky_k>2n+1$.
\end{proof}

This implies that any $T_n$ for which the associated equation \eqref{pellova_m} has a solution is a member of infinitely many $D(a)$-triples $\{T_n,T_{m_k},T_{m_{k+1}}\}$ for $k\geq 1$ such that $n\notin\{m_k,m_{k+1}\}$ (since the sequence $(m_k)_{k\geq 1}$ is strictly increasing by \eqref{niz_y_x}, this excludes at most two values of $k$). Thus, we conclude the proof of Theorem \ref{tm_2}.

\begin{rem}
If $(x^*,y^*)$ is any solution of \eqref{pellova_x} satisfying \eqref{ineq_granice_fund}, then $|x^*|\leq ny^*$ if $N<0$ (by the bound for $y^*$), and $|x^*|\geq (n+1)y^*$ if $N>0$ (by the bound for $x^*$). Hence, by induction, using \eqref{niz_y_x} and $x_{k+1}=(2n+1)x_k+2n(n+1)y_k$, for $k\geq 1$ the numbers $x_k$ and $y_k$ are both positive, unless $N>0$ and $x^*<0$, in which case both are negative; in particular, $|y_{k+1}|>(2n+1)|y_k|$. Since the assumption $x^*>0$ for $N>0$ enters the proof of Proposition~\ref{prop_mk} only through $x_ky_k>0$, Lemma~\ref{lem_y} and Proposition~\ref{prop_mk} hold without it, if we set $m_k=(|y_k|-1)/2$.
\end{rem}

\section{Existence of $D(a)$-pairs}\label{section:ex_pairs}

In this section, we will explore for which  integers $a\neq 0$ the equation \eqref{pellova_x} can have a solution.

\begin{theorem}\label{thm:modular_9}
Let $a\neq 0$ be an integer such that there exists a $D(a)$-pair of triangular numbers. Then $a \not\equiv 2, 5 \pmod 9$.
\end{theorem}
\begin{proof}
    Let $r$ be a nonnegative integer such that
    $$T_n\cdot T_m+a=r^2.$$
    Then $r^2\equiv 0,1,4,7\pmod 9$ and $T_n\cdot T_m\equiv 0,1,3,6\pmod 9$. From $a\equiv r^2- T_n\cdot T_m$, by considering all possibilities, we get
    $a\equiv 0,1,3,4,6,7,8\pmod 9.$
\end{proof}

We will now prove that the converse does not hold. For each remaining class modulo $9$, we give an example of an integer $a$ that admits no $D(a)$-pair of triangular numbers. For illustrative purposes, we detail the proofs for $a=12$ and $a=-37$ ($a\equiv 3,8\pmod 9$). In the cases $a\in\{252,487,-158,42,52\}$, i.e., $a\equiv 0,1,4,6,7\pmod 9$, the proofs proceed analogously and are omitted for brevity.

\begin{proposition} \label{prop:nonexist}
Let $a \in \{12, -37\}$. For any positive integer $n$, the equation \eqref{pellova_x}
does not have any integer solutions $(x, y)$ such that $y$ is odd. Consequently, there exists no pair of triangular numbers $\{T_n, T_m\}$ with the property $D(12)$ or $D(-37)$.
\end{proposition}

\begin{proof}
Assume, for the sake of contradiction, that there exists a positive integer $n$ such that equation \eqref{pellova_x} has an integer solution $(x, y)$ with an odd integer $y \ge 3$.
By setting $y = 2m+1$ for some positive integer $m$, equation \eqref{pellova_x} can be rewritten as:
$$ x^2 = 16a + n(n+1)((2m+1)^2-1). $$

Notice that $4n(n+1)=(2n+1)^2-1$. Hence, after multiplying the equation by $4$, we see that the symmetric equation
\begin{equation} \label{eq:sym_combined_mixed}
c^2 = (U^2-1)(V^2-1) + 64a
\end{equation}
in a nonnegative integer $c$ and odd positive integers $U$ and $V$ has the solution $(c,U,V)=(2|x|,2n+1,2m+1)$, and, by symmetry, also the solution $(2|x|,2m+1,2n+1)$.
In what follows, we restrict our analysis of equation \eqref{eq:sym_combined_mixed} to solutions with $V\geq U$.

If $V = 1$ (or $U = 1$), then  $c^2 = 64a$.
 For $a=12$, this yields $c^2 = 768$, which is not a perfect square.
    For $a=-37$, this yields $c^2 = -2368$, which is trivially impossible.

Hence, any valid solution to \eqref{eq:sym_combined_mixed} requires $U \ge 3$ and $V \ge 3$.

Equation \eqref{eq:sym_combined_mixed} can be rewritten as a generalized Pellian equation in variables $X=c$ and $Y=V$:
$$ X^2 - (U^2-1)Y^2 = 64a - (U^2-1). $$
By our assumption and the well-ordering principle, among the solutions of \eqref{eq:sym_combined_mixed} with $V\geq U\geq 3$ there is one for which the sum $U+V$ is minimal. Let $(c_0,U_0,V_0)$ be this minimal solution. The fundamental solution $(U_0, 1)$ to the associated Pell equation $X^2 - (U_0^2-1)Y^2 = 1$ allows us to define a new solution $(c', V')$:
$$ (c', V') := \bigl(U_0 c_0 - (U_0^2-1)V_0,\;\; U_0 V_0 - c_0\bigr). $$
Observe that the parity of the new elements $(c', V')$ is preserved, and that $(|c'|,U_0,|V'|)$ is also a solution of \eqref{eq:sym_combined_mixed}.
Furthermore, $|V'| \neq 1$; otherwise, we would have either $64a=(V_0-U_0)^2$ or $64a=(V_0+U_0)^2$, which occurs only when $a$ is a perfect square. Since $V'$ is odd, we conclude that $|V'| \ge 3$.

For the new solution $(c', |V'|)$ to satisfy the  inequality $|V'| < V_0$, the value of $c_0$ must fall strictly between $(U_0-1)V_0$ and $(U_0+1)V_0$. Squaring this condition yields two inequalities that must hold simultaneously:
\begin{align}
    c_0^2 &< (U_0+1)^2 V_0^2 \iff 2V_0^2(U_0+1) > 64a + 1 - U_0^2, \label{upper_bound} \\
    c_0^2 &> (U_0-1)^2 V_0^2 \iff 2V_0^2(U_0-1) > -64a - 1 + U_0^2. \label{lower_bound}
\end{align}
If both bounds hold, we obtain a new integer solution with $3 \le |V'| < V_0$. Then $(|c'|,U_0,|V'|)$, or $(|c'|,|V'|,U_0)$ if $|V'|<U_0$, is a solution of \eqref{eq:sym_combined_mixed} with $V\geq U\geq 3$ and with a smaller sum, which contradicts the minimality of $U_0+V_0$. Therefore,  at least one of these inequalities must fail. This restricts the minimal solution to a finite set of exceptional pairs $(U_0, V_0)$.

\textbf{Case $a=12$:} For a positive parameter $a$, \eqref{lower_bound} holds trivially. Inequality \eqref{upper_bound} fails when $2V_0^2(U_0+1) \le 769 - U_0^2$. This yields exactly six pairs with $V_0 \ge U_0 \ge 3$: $(3,3)$, $(3,5)$, $(3,7)$, $(3,9)$, $(5,5)$, and $(5,7)$. Substituting these into \eqref{eq:sym_combined_mixed} generates values for $c^2$ which are not perfect squares.

 \textbf{Case $a=-37$:} For a negative parameter $a$, inequality \eqref{upper_bound} holds trivially. The bound \eqref{lower_bound} does not hold when $2V_0^2(U_0-1) \le 2367 + U_0^2$. From equation \eqref{eq:sym_combined_mixed}, the condition $c_0^2 \ge 0$ implies that $(U_0^2 - 1) (V_0^2 - 1) \geq 2368$. Combining these two inequalities yields 13 pairs: $(3,19)$, $(3,21)$, $(3,23)$, $(5,11)$, $(5,13)$, $(5,15)$, $(5,17)$, $(7,9)$, $(7,11)$, $(7,13)$, $(9,9)$, $(9,11)$, and $(11,11)$. However, none of these pairs yields a perfect square for $c^2$.

Since the supposed minimal solution $(c_0,U_0,V_0)$ must belong to this set of pairs, and none of these pairs yields a perfect square for $c^2$, such a minimal solution cannot exist. Consequently, the equation has no solutions for $a=12$ or $a=-37$.
\end{proof}

\begin{rem}
It is worth noting that applying the same algorithm for $a=-1$ leads to the pair $(U_0, V_0) = (3,3)$. Substituting these values into equation \eqref{eq:sym_combined_mixed} yields $c_0=0$. This pair corresponds to $n=1$ and $m=1$, which exactly produces the fundamental solution $(x^*, y^*) = (0,3)$ for the initial generalized Pellian equation \eqref{pellova_x}. Since $m=n=1$, this trivial case does not yield a $D(-1)$-pair of distinct triangular numbers. However, we can generate an index $m>n$ by applying the recurrence relation \eqref{eq_niz_mk}, and we obtain $m_1 = 4$, i.e., the $D(-1)$-pair $\{T_1, T_4\}$.
\end{rem}

While it is known from \cite{mbt_triangular} that $D(a)$-pairs of triangular numbers exist whenever $a$ is a perfect square, we now construct additional infinite families of integers $a$ that admit such pairs.

Using the well-known property that the sum of two consecutive triangular numbers is a perfect square, $T_{k+1} + T_k = (k+1)^2$, $k\geq 1$, we can construct explicit $D(a)$-pairs by fixing $a$ in relation to the sequence of triangular numbers. Thus the next proposition gives us infinite families of both positive and negative numbers $a$ for which we know $D(a)$-pairs of triangular numbers exist.

\begin{proposition} \label{prop:a_is_triangular}
Let $k \ge 1$ be an integer.
\begin{enumerate}
    \item If $a = T_k$, then the set $\{T_1, T_{k+1}\}$ is a $D(a)$-pair of triangular numbers.
    \item If $a = 1 - T_k$ and $k\geq 2$, then the set $\{T_2, T_{k+1}\}$ is a $D(a)$-pair of triangular numbers.
\end{enumerate}
\end{proposition}
\begin{proof}
Let $k\geq 1$ and $a = T_k$. Since $T_1 = 1$, we have:
$$
T_1 \cdot T_{k+1} + a = 1 \cdot \frac{(k+1)(k+2)}{2} + \frac{k(k+1)}{2}= (k+1)^2.
$$
Thus, $\{T_1, T_{k+1}\}$ forms a $D(T_k)$-pair.

Now, let $a = 1 - T_k$. Substituting $T_2 = 3$, we have:
$$ T_2 \cdot T_{k+1} + a = 3 \cdot \frac{(k+1)(k+2)}{2} + 1 - \frac{k(k+1)}{2}=(k+2)^2. $$
Thus, $\{T_2, T_{k+1}\}$ is a $D(1 - T_k)$-pair for all $k \ge 2$.
\end{proof}

\section{Finding triangular numbers which form a $D(a)$-pair}\label{section:member_pairs}

In \cite{mbt_triangular} it has been shown that for every integer $n\geq 1$ and $a=k^2$, $k\geq 1$, triangular number $T_n$ is a member of a $D(a)$-pair of triangular numbers. It is natural to explore this further and for a fixed $a\neq0$, ask which triangular numbers $T_n$ can be members of a $D(a)$-pair of triangular numbers.

\begin{theorem} \label{thm:generalized}
Let $a$ be a non-zero integer. The triangular number $T_n$ can be a member of a $D(a)$-pair $\{T_n,T_m\}$, $m\geq 1$, only if all prime factors of $n(n+1)$ belong to the following set: the prime $2$, prime divisors of $a$, and odd primes $p$ for which $a$ is a quadratic residue modulo $p$.
\end{theorem}
\begin{proof}
Assume that $n$ is a positive integer such that there exists a $D(a)$-pair $\{T_n,T_m\}$. Then   integers $x=4r$ and $y=2m+1$ are solutions of the equation \eqref{pellova_x}. Let $p$ be an arbitrary odd prime divisor of $n(n+1)$.

After reducing equation \eqref{pellova_x} modulo $p$, we have:
$$ x^2 \equiv 16a \pmod p. $$
If $p \mid a$, the congruence trivially becomes $x^2 \equiv 0 \pmod p$, which is satisfied for all $x$ such that $p \mid x$.

If $p \nmid a$, then $16a$ is not divisible by $p$. We can divide the congruence by $16$, which gives:
$$ r^2=(x/ 4)^2 \equiv a \pmod p. $$
This implies that $a$ is a quadratic residue modulo $p$.
\end{proof}

Let's illustrate this condition further on $a=-1$.

\begin{kor}\label{kor:minus1}
Let $n$ be a positive integer such that there exists a positive integer $m$ for which $\{T_n,T_m\}$ is a $D(-1)$-pair. Then both $n$ and $n+1$ are numbers of the form $2^e\cdot p_1^{\alpha_1}\cdots p_l^{\alpha_l}$, where $p_i$ are prime numbers and $p_i\equiv 1\pmod{4}$ holds for each $i\in\{1,\dots,l\}$, $\alpha_1,\dots,\alpha_l$ are positive integers and $e\in\{0,2\}$ for $n$ and $e\in\{0,1\}$ for $n+1$.
\end{kor}
\begin{proof}
Let $a=-1$. From Theorem \ref{thm:generalized} we know that any odd prime factor $p$ of either $n$ or $n+1$ must be such that $-1$ is a quadratic residue modulo $p$. Hence, $p$ must be congruent to $1 \pmod{4}$.

Consequently, neither $n$ nor $n+1$ can be congruent to $3 \pmod{4}$. Since $n$ and $n+1$ are consecutive integers, the only possible residue classes for the pair $(n, n+1)$ modulo $4$ are $(1, 2)$ or $(0, 1)$. This immediately eliminates $e=1$ for $n$ (since $n \not\equiv 2 \pmod{4}$) and $e \geq 2$ for $n+1$ (since $n+1 \not\equiv 0 \pmod{4}$). Thus, $e \in \{0, 1\}$ for $n+1$, and $e=0$ or $e \geq 2$ for $n$.

It remains to show that $e \leq 2$ for $n$. Assume, for the sake of contradiction, that $e \geq 3$, which implies $n \equiv 0 \pmod{8}$. From the definition of a $D(-1)$-pair, we have
$$ T_n T_m = r^2 + 1. $$
If $n \equiv 0 \pmod{8}$, then $T_n=\frac{n(n+1)}{2}$ is divisible by $4$. This implies that $r^2 \equiv 3 \pmod{4}$, which is impossible for any integer $r$. Thus, $n$ cannot be a multiple of $8$, proving $e \in \{0, 2\}$ for $n$.
\end{proof}

\begin{rem}
 Let us show that the converse of Corollary \ref{kor:minus1}, and consequently Theorem \ref{thm:generalized}, does not hold. Number $n=100$ satisfies the necessary conditions of Corollary \ref{kor:minus1}. But equation \eqref{pellova_x} has no solution $(x^*,y^*)$ satisfying the bounds \eqref{ineq_granice_fund}. Hence, the triangular number $T_{100}$ is not a member of any $D(-1)$-pair containing only triangular numbers. It can, however, form a $D(-1)$-pair with a second element that is not a triangular number, e.g., $\{T_{100}, 5653\}$ is a $D(-1)$-pair.
\end{rem}

Observe that from inequalities \eqref{ineq_granice_fund} and the fact that $(X_1,Y_1)=(2n+1,2)$, it follows that $y^*<2n+1$ for $n\geq \sqrt[3]{4|a|}$. From this observation and Lemma \ref{lema_svojstva_xy} we see that the fundamental solution $(x^*,y^*)$ of the class containing the solution $(4r,2m+1)$ of \eqref{pellova_x} which corresponds to a $D(a)$-pair $\{T_n,T_m\}$, with odd integer $y^*\geq 3$, yields a positive integer $m'=(y^*-1)/2$ strictly smaller than $n$ such that $\{T_{m'},T_n\}$ is a $D(a)$-pair. This proves the next result.

  \begin{proposition}\label{prop:descent}
    Let $a\neq 0$ be an integer that is not a perfect square.
    If a triangular number $T_n$, with $n \geq \sqrt[3]{4|a|}$, is a member of a $D(a)$-pair of triangular numbers, then there exists an integer $n_0$, $1 \leq n_0 < n$, such that $\{T_{n_0},T_n\}$ is a $D(a)$-pair, i.e., the pair {$(x,2n+1)$} is a solution to the equation {$x^2 - n_0(n_0+1)y^2 = 16a - n_0(n_0+1)$} for some integer $x$.
\end{proposition}

We exploit Proposition~\ref{prop:descent} to describe an algorithm which, for a given integer $a\neq 0$ that is not a perfect square and a given bound $C$, finds all $n\leq C$ such that $T_n$ is a member of a $D(a)$-pair of triangular numbers. For $n\geq 1$, let $F_n$ be the set of all solutions $(x^*,y^*)$ of \eqref{pellova_x} satisfying \eqref{ineq_granice_fund} with $y^*$ odd. The set $F_n$ is finite and can be computed directly: for each odd $y^*$ satisfying the bound in \eqref{ineq_granice_fund}, we check whether $N+n(n+1)y^{*2}$ is the square of an integer $x^*$ satisfying the bound in \eqref{ineq_granice_fund}, and if it is, then $(\pm x^*,y^*)\in F_n$.
\begin{enumerate}[label=(\roman*), leftmargin=*]
\item Let $S=\{n\leq C : n<\sqrt[3]{4|a|},\ F_n\neq\emptyset\}$.
\item Let $n_0$ be the smallest unprocessed element of $S$. For each $(x^*,y^*)\in F_{n_0}$, compute $y_k$, $k\geq 0$, by \eqref{eq_rek_y} with $n=n_0$ until the first $k\geq 1$ such that $|y_k|>2C+1$, and add to $S$ all integers $m=(|y_k|-1)/2$ such that $n_0<m\leq C$. Mark $n_0$ as processed.
\item Repeat step (ii) until all elements of $S$ are processed. Return $S$.
\end{enumerate}
Step (i) is justified by Lemma~\ref{lema_svojstva_xy} and the proof of Theorem~\ref{tm_2}, since ${F_n\neq\emptyset}$ if and only if $T_n$ is a member of a $D(a)$-pair. In step (ii), every $D(a)$-pair $\{T_{n_0},T_m\}$ gives a solution $(4r,2m+1)$ of \eqref{pellova_x} with $n=n_0$, which equals $\pm(x_k,y_k)$ for some $k\in\mathbb{Z}$, where $(x_k,y_k)$ is defined as in Section~\ref{section:triple} (extended to $k<0$) from the fundamental solution $(x^*,y^*)\in F_{n_0}$ of its class, thus ${2m+1=|y_k|}$. By conjugation, $y_{-k}$ equals the term with index $k$ computed from $(-x^*,y^*)\in F_{n_0}$, so $k\geq 0$ suffices, and by the remark at the end of Section~\ref{section:triple}, $|y_k|$ is increasing for $k\geq 1$, which justifies the stopping rule. Hence step (ii) adds to $S$ exactly the integers $m\in(n_0,C]$ for which $\{T_{n_0},T_m\}$ is a $D(a)$-pair, and the algorithm is correct by Proposition~\ref{prop:descent} and induction on $n$.

\begin{example}
Let us set the upper bound to $C=1000$ and search for indices $n$ of members of $D(-1)$-pairs of triangular numbers. For $a=-1$ we have $\sqrt[3]{4}<2$, so step (i) gives $S=\{1\}$, since $(x^*,y^*)=(0,3)$ is a solution for $n=1$. Starting with the base case $n=1$, we recursively generate the solutions.

\begin{table}[h!]
\centering
\caption{Generation of indices $m \leq 1000$ }
\label{table_iteracije}
\renewcommand{\arraystretch}{1.3}
\begin{tabular}{|c|c|l|}
\hline
\textbf{Step} & \textbf{Seed $n$} & \textbf{New solutions $m \in (n, 1000]$} \\
\hline
1 & $1$ & $4, 25, 148, 865$ \\
\hline
2 & $4$ & $25, 457$ \\
\hline
3 & $25$ & $148, 457$ \\
\hline
4 & $148$ & $865$ \\
\hline
5 & $457$ & None \\
\hline
6 & $865$ & None \\
\hline
\end{tabular}
\end{table}

The algorithm naturally terminates at this point, as applying the recurrence relations to $n=457$ and $n=865$ yields solutions that exceed the bound of $1000$. Consequently, the complete set of numbers $n$ up to $1000$ such that $T_n$ is a member of a $D(-1)$-pair of triangular numbers is $\{1, 4, 25, 148, 457, 865\}$, as listed in Table \ref{table_iteracije}.
\end{example}


\section*{Declarations}

\textbf{Funding Declaration} \\
The author is supported by the Croatian Science Foundation grant no.~IP-2022-10-5008 and by the University of Split, grant no.~IP-UNIST-44, funded by the European Union -- NextGenerationEU.

\vspace{1em}

\textbf{Use of Artificial Intelligence} \\
Gemini 3.1 Pro was used to explore research directions and to refine the language. Claude Opus 4.8 generated an initial outline of the proof of Proposition~\ref{prop:nonexist}, which the author corrected, formalized and verified. During the revision, Claude Opus 5.5 was used to check the correctness of the proofs, to point out gaps and suggest corrections, to improve the presentation and to verify the computations numerically. The main ideas and results are the author's own work, and the author assumes full responsibility for the manuscript.

\end{document}